\documentclass[11pt]{amsart}

\usepackage[margin=1.0in]{geometry}

\usepackage{amsmath,amssymb,amsthm}
\usepackage[hidelinks]{hyperref}

\numberwithin{equation}{section}

\newtheorem{theorem}{Theorem}[section]
\newtheorem{proposition}[theorem]{Proposition}
\newtheorem{lemma}[theorem]{Lemma}

\newcommand{\R}{\mathbb R}
\newcommand{\E}{\mathbb E}
\newcommand{\Prob}{\mathbb P}
\newcommand{\cE}{\mathcal E}
\newcommand{\cH}{\mathcal H}
\newcommand{\cN}{\mathcal N}
\newcommand{\ip}[2]{\langle #1,#2\rangle}
\newcommand{\norm}[1]{\lVert #1\rVert}
\newcommand{\one}{\mathbf 1}
\newcommand{\Qn}{[-1,1]^n}
\newcommand{\Vn}{\{-1,1\}^n}
\newcommand{\Sn}{S^{n-1}}
\newcommand{\conv}{\operatorname{conv}}
\newcommand{\Var}{\operatorname{Var}}
\newcommand{\rate}{\mathcal I}

\title[Log-Free Facets of $0/1$-Polytopes]
{A Log-Free Lower Bound for the Number of Facets of $0/1$-Polytopes}
\author{Omer Friedland}
\address{Institut de Math\'ematiques de Jussieu-Paris Rive Gauche, Sorbonne Universit\'e, Paris, France}
\email{omer.friedland@sorbonne-universite.fr}
\date{}


\begin{document}

\begin{abstract}
Let $g(n)$ denote the largest number of facets of a full-dimensional
$0/1$-polytope in $\R^n$. We prove that there are absolute constants
$c>0$ and $n_0$ such that
$$
g(n)\ge(cn)^{n/2}\quad(n\ge n_0).
$$
This removes the logarithmic factor from the lower bound
$(cn/\log n)^{n/2}$ of Gatzouras, Giannopoulos, and Markoulakis. The proof
compares a random sign polytope with two Rademacher rate bodies separated by
a fixed level gap. On a flat patch of the outer body, nonpenetrating facets
have uniformly bounded footprints, so the main obstruction comes from facets
whose outer halfspaces penetrate the inner body. Such a facet forces an empty
buffered discrete cap whose mass is controlled by its penetration depth. For
shallow penetration, likelihood localization confines a comparable part of
the cap to a narrow slab, where a conditional $\varepsilon$-net argument
controls empty halfspace ranges. For deep penetration, the cap is large enough
for a global discretization. These two regimes make the visible penetrating
area exponentially negligible relative to the uncovered outer area, and the
result follows by dividing by the maximal footprint of a nonpenetrating facet.
\end{abstract}

\maketitle

\section{Introduction}

For a full-dimensional polytope $P\subset\R^n$, let $f_{n-1}(P)$ denote
its number of facets, and set
$$
g(n) := \max\{f_{n-1}(P):P = \conv V, V\subset\{0,1\}^n, \dim P = n\}.
$$
The problem of estimating the growth of $g(n)$ was raised by Fukuda and
Ziegler; see \cite{Ziegler}. Fleiner, Kaibel, and Rote proved the factorial
upper bound
$$
g(n)\le 30(n-2)!
$$
for all sufficiently large $n$ \cite{FKR}. The first superexponential lower
bound was obtained by B\'ar\'any and P\'or, who showed that
$$
g(n)\ge\left(\frac{cn}{\log n}\right)^{n/4}
$$
for an absolute constant $c>0$ \cite{BP}. Building on the random-polytope
method originating in the work of Dyer, F\"uredi, and McDiarmid
\cite{DFM}, Gatzouras, Giannopoulos, and Markoulakis improved this first to
$$
g(n)\ge\left(\frac{cn}{\log^2 n}\right)^{n/2}
$$
and then to
$$
g(n)\ge\left(\frac{cn}{\log n}\right)^{n/2};
$$
see \cite{GGM05,GGM07}. The following result removes the remaining logarithmic factor.

\begin{theorem}\label{thm:main}
There are absolute constants $c>0$ and $n_0\in\mathbb N$ such that
$$
g(n)\ge(cn)^{n/2}
$$
for every $n\ge n_0$.
\end{theorem}

In logarithmic form, Theorem~\ref{thm:main} gives
$$
\log g(n)\ge\frac n2\log n-O(n),
$$
whereas the factorial upper bound gives
$$
\log g(n)\le n\log n-n+O(\log n).
$$
Thus, at the time of the first version of this paper, the leading coefficient of $n\log n$ remained between $1/2$ and $1$. Subsequently, Castillo and Ferroni \cite{CF26} obtained a substantially stronger deterministic lower bound which implies
$$
\log g(n)\ge n\log n-n-O((\log n)^2).
$$
Together with the known upper bound, their result yields
$$
\log g(n)=n\log n-n+O((\log n)^2).
$$
Their result therefore quantitatively supersedes Theorem~\ref{thm:main}. The present paper gives a different probabilistic approach, continuing the random-polytope method of \cite{BP,GGM05,GGM07}.

\subsection*{Proof strategy}

We pass affinely from the $0/1$ cube to the sign cube $\Vn = \{-1,1\}^n$.
For the Rademacher rate function
$$
I(s) = \frac12((1+s)\log(1+s)+(1-s)\log(1-s)), \quad \rate(x) = \sum_{i = 1}^n I(x_i),
$$
let $\cE_t = \{x\in[-1,1]^n:\rate(x)\le t\}$. We take
$N = \lfloor e^{\alpha n}\rfloor$ independent uniform sign vectors and let
$P_N$ be their convex hull. The relevant levels are
$$
t_\pm = \alpha n-\frac32\log n\pm B,
$$
where $\alpha>0$ is small and $B>0$ is a sufficiently large fixed
constant.

A flat patch of $\partial\cE_{t_+}$ has area of order
$c^n n^{(n-1)/2}|S^{n-1}|$. The probability of a tangent discrete cap is
at most $Ce^{-t_+}/\sqrt n$, so an $\exp(-Cne^{-B})$ fraction of this
patch remains outside $P_N$ in expectation. A facet whose outer halfspace
misses the truncated inner body $\cE_{t_{-}}\cap\gamma[-1,1]^n$ has footprint
at most $(6B)^{(n-1)/2}|S^{n-1}|$ on the outer patch. Consequently, the
only obstruction to dividing uncovered area by this footprint is a facet
whose outer halfspace enters the inner body; we call such a facet
\emph{penetrating}.

A penetrating halfspace is assigned a minimum rate level $t_H$ and a
buffered cap $C_H\subset\Vn$. If its penetration depth is
$s = t_{-}-t_H$, then
$$
N\mu_n(C_H)\gtrsim n\frac{e^{B+s}}{(1+B+s)^{3/2}}.
$$
For shallow penetration, this occupancy is too small to absorb a global
normal-net cost of order $n\log n$. The main new estimate places a
comparable part of $C_H$ inside a narrow likelihood slab. Conditioning on
that slab reduces the trace entropy, and a conditional $\varepsilon$-net
argument, in the spirit of \cite{HW}, excludes shallow penetrating facets.
For $s\ge2\log\log n$, the cap occupancy dominates the global entropy, so
a direct discretization of normals and thresholds excludes deep penetrating
facets. The expected area visible through penetrating facets is therefore
negligible compared with the uncovered area, and the required facet count
follows.

The proof of Theorem~\ref{thm:main} is completed in
Section~\ref{sec:main-proof}, assuming three quantitative propositions.
The outer-patch estimate is proved in Section~\ref{sec:outer-proof};
penetrating facets are treated in Section~\ref{sec:penetrating}; and the
buffered-cap estimate is proved in Section~\ref{sec:cap-proof}. The two
sampling estimates are included in Appendix~\ref{sec:sampling}.

\subsection*{Disclosure of AI assistance}
During the development of the final argument, the author used ChatGPT, developed by OpenAI, to explore and test possible approaches for closing the argument, and this assistance helped complete the proof. ChatGPT also provided substantial assistance with the organization and drafting of the manuscript. The problem, the overall project, and substantial prior progress, including several partial approaches and intermediate drafts, predate this AI-assisted phase. The author independently verified all mathematical statements and arguments and takes full responsibility for the paper.

\section{Rate bodies and basic inputs}

All logarithms are natural. Constants denoted by $c,C>0$ may change from
line to line. Subscripted constants remain fixed within the argument in which
they are introduced. Parameters are chosen in the order
$$
\gamma,\quad \alpha,\quad \text{ the constants in the auxiliary estimates},\quad B,\quad n_0.
$$
After these choices, every constant is absolute. The affine map
$x\mapsto(x+\one)/2$ preserves face lattices, so it is enough to construct
a full-dimensional polytope with vertices in $\Vn$.

Let $\mu_A$ be the uniform probability measure on $\{-1,1\}^A$. For
$y = (y_i)_{i\in A}\in(-1,1)^A$, let $\Prob_y^A$ be the product law with
coordinate means $y_i$, and write $\E_y^A$ and $\Var_y^A$ for expectation
and variance under $\Prob_y^A$. Define
$$
I(s) := \frac12((1+s)\log(1+s)+(1-s)\log(1-s)), \quad -1\le s\le1,
$$
with $0\log0 := 0$, and put
$$
\lambda(s) := I'(s) = \operatorname{artanh}s, \quad I''(s) = \frac1{1-s^2}\ge1.
$$
For $y\in(-1,1)^A$ and $z\in\{-1,1\}^A$, write
$$
\rate_A(y) := \sum_{i\in A}I(y_i), \quad \ell_y^A(z) := \sum_{i\in A}\lambda(y_i)(z_i-y_i).
$$
When $A = [n]$, the superscript is omitted; when $A = \{1,\dots,r\}$, we
write $\mu_r$, $\Prob_y$, $\E_y$, and $\Var_y$. Empty coordinate blocks
are interpreted in the natural way. The likelihood identity is
\begin{align}
\frac{d\mu_A}{d\Prob_y^A}(z) = \exp\{-\rate_A(y)-\ell_y^A(z)\}.
\end{align}
For $0<t<n\log2$, define the rate body
$$
\cE_t := \{x\in\Qn:\rate(x)\le t\}.
$$
Let $\cH^{n-1}$ denote $(n-1)$-dimensional Hausdorff measure and let
$\Phi$ be the standard normal distribution function. For a convex body $K$,
let $S_K$ be its surface-area measure on $\Sn$, and write
$$
h_K(v) := \sup_{z\in K}\ip{v}{z}, \quad \omega_{n-1} := \cH^{n-1}(\Sn).
$$
For a closed convex set $K$ and $y\in K$, our normal-cone convention is
$$
N_K(y) := \{v:\ip{v}{z-y}\le0\text{ for every }z\in K\}.
$$
The function $\rate$ is smooth and strictly convex on $(-1,1)^n$.

We use the following rate-body geometry from
Gatzouras, Giannopoulos, and Markoulakis \cite[Lemma~2.4 and the proof of
Proposition~3.1]{GGM07}. Their level is $\rate/n$; the statement below is
written in the unnormalized level $t = \rate$. In their notation, the normal
set used in the proof of Proposition~3.1 is the set $\Sigma$ below. The
support-point containment and the area estimate follow from that proof by
combining its flat-direction measure bound with its curvature estimate.

\begin{lemma}[Rate-body geometry]\label{lem:geometry}
There are absolute constants $\gamma_*\in(0,1)$ and $r>1$ with the following
properties.

\emph{(i)} If $0<t<t+\Delta<n\log2$, $0<\gamma\le\gamma_*$, and a closed
halfspace $H$ satisfies
$$
\operatorname{int}H\cap(\cE_t\cap\gamma\Qn) = \varnothing,
$$
then
\begin{align}
\cH^{n-1}(\partial\cE_{t+\Delta}\cap\gamma\Qn\cap H) \le(3\Delta)^{(n-1)/2}\omega_{n-1}.
\end{align}

\emph{(ii)} Put
$$
\Sigma := \left\{u\in\Sn:\norm{u}_\infty\le\sqrt{\frac rn}\right\}.
$$
For every fixed $\eta\in(0,\gamma_*)$ there is $b_\eta>0$ such that, for
all sufficiently large $n$ and all $0<t\le b_\eta n$, the unique support
point $x(u,t)$ of $\cE_t$ in every direction $u\in\Sigma$ belongs to
$\eta\Qn$, the map $u\mapsto x(u,t)$ is continuous on $\Sigma$, and
\begin{align}
S_{\cE_t}(\Sigma) \ge e^{-n/2}(1-\eta^2)^{n-1}(2t)^{(n-1)/2}\omega_{n-1}.
\end{align}
\end{lemma}

We also use Berry-Esseen in the form
\begin{align}
\sup_s\left|\Prob\left\{\frac{\sum_i Z_i}{\sigma}\le s\right\} -\Phi(s)\right| \le \frac{C_{\rm BE}}{\sigma^3}\sum_i\E|Z_i|^3, \quad \sigma^2 = \sum_i\E Z_i^2>0,
\end{align}
for independent centered variables $Z_i$; see \cite[Chapter~V]{Petrov}.
Finally, we use the one-sided Montgomery-Smith estimate \cite{MS}: there
is a universal $c_{\rm MS}\ge1$ such that
\begin{align}\label{eq:MS}
\Prob\left\{\sum_i a_iX_i \ge c_{\rm MS}^{-1}t\norm a_2\right\} \ge c_{\rm MS}^{-1}e^{-c_{\rm MS}t^2}
\end{align}
whenever $a$ is a nonzero finite real vector, the $X_i$ are independent
uniform signs, and $0<t\le\norm a_2/\norm a_\infty$.

\section{The proof reduced to penetrating facets}\label{sec:main-proof}

The proof uses two rate levels separated by a fixed amount. The outer level
has a large boundary patch which is not completely covered by a random sign
polytope. A facet whose outer halfspace misses the inner level sees only a
bounded part of that patch. The sole obstruction is therefore a facet whose
outer halfspace enters the inner body; such facets will be called
\emph{penetrating}.

The first proposition supplies the outer patch and the probability estimates
used throughout the proof.

\begin{proposition}[Outer patch]\label{prop:outer}
There are absolute constants
$$
0<\gamma_0<\gamma<\gamma_*,\quad \alpha>0, \quad c_0,C_0,C_{\rm out}>0
$$
with the following property. Fix $B\ge1$ and, for all sufficiently large
$n$, set
\begin{align}\label{eq:levels}
N := \lfloor e^{\alpha n}\rfloor, \quad t_\pm := \alpha n-\frac32\log n\pm B, \quad E_\pm := \cE_{t_\pm}\cap\gamma\Qn.
\end{align}
Let $X^{(1)},\dots,X^{(N)}$ be independent uniform points of $\Vn$ and put
$$
P_N := \conv\{X^{(1)},\dots,X^{(N)}\}.
$$
For $u\in\Sigma$, let $x(u) := x(u,t_+)$ be the support point of $\cE_{t_+}$
in direction $u$, and set
$$
\Gamma_+ := \{x(u):u\in\Sigma\}, \quad A_+ := S_{E_+}(\Sigma) = \cH^{n-1}(\Gamma_+).
$$
Then $x(u)\in\gamma_0\Qn$ for every $u\in\Sigma$, and
\begin{align}\label{eq:outer-area}
A_+\ge c_0^n n^{(n-1)/2}\omega_{n-1}.
\end{align}
Moreover, uniformly for $u\in\Sigma$,
\begin{align}
\mu_n\{z:\ip{u}{z}\ge\ip{u}{x(u)}\} &\le C_0\frac{e^{-t_+}}{\sqrt n}, \label{eq:outer-cap} \\
\mu_n\{z:|\ell_{x(u)}(z)|\le L\} &\le C_0\frac{(1+L)e^{L-t_+}}{\sqrt n} \quad(1\le L\le\sqrt n). \label{eq:outer-window}
\end{align}
If
$$
\Omega_{\rm out} := \{u\in\Sigma:x(u)\notin P_N\}, \quad A_{\rm out} := S_{E_+}(\Omega_{\rm out}),
$$
then
\begin{align}\label{eq:uncovered}
\E A_{\rm out}\ge q_{B,n}A_+, \quad q_{B,n} := \exp(-C_{\rm out}ne^{-B}).
\end{align}
\end{proposition}

Fix the constants in Proposition~\ref{prop:outer}. Let
$\mathcal D := \{\dim P_N = n\}$. On $\mathcal D$, orient each facet $Q$ of
$P_N$ by its outer unit normal $\nu_Q$ and write
$$
b_Q := h_{P_N}(\nu_Q), \quad H_Q^+ := \{z:\ip{\nu_Q}{z}\ge b_Q\}.
$$
The facet $Q$ is \emph{penetrating} if
$$
\operatorname{int}H_Q^+\cap E_{-}\ne\varnothing,
$$
and it is \emph{visible from} $u\in\Sigma$ if $x(u)\in H_Q^+$. Define
$$
\Omega_{\rm pen} := \{u\in\Sigma:\text{ some penetrating facet is visible from }u\}.
$$
On $\mathcal D$, set $A_{\rm pen} := S_{E_+}(\Omega_{\rm pen})$. On
$\mathcal D^c$, set $\Omega_{\rm pen} := \varnothing$ and
$A_{\rm pen} := A_+$.

\begin{proposition}[Counting reduction]
On $\mathcal D$,
\begin{align}\label{eq:facet-cover}
A_{\rm out} \le A_{\rm pen} +f_{n-1}(P_N)(6B)^{(n-1)/2}\omega_{n-1}.
\end{align}
\end{proposition}

\begin{proof}
By the normal-area identity \eqref{eq:normal-area}, every point of
$\Gamma_+\setminus P_N$ violates a facet inequality, and the part visible
through penetrating facets has area at most $A_{\rm pen}$. If
$Q$ is not penetrating, then
$\operatorname{int}H_Q^+\cap E_{-} = \varnothing$. Applying
Lemma~\ref{lem:geometry}(i) with $t = t_{-}$ and
$\Delta = t_+-t_{-} = 2B$ bounds the footprint of $Q$ on $\Gamma_+$ by
$(6B)^{(n-1)/2}\omega_{n-1}$. Summing over the nonpenetrating facets proves
\eqref{eq:facet-cover}.
\end{proof}

The remaining estimate is pointwise in the outer normal.

\begin{proposition}[Penetrating visibility]\label{prop:visibility}
The constant $B$ can be chosen so that, for some absolute
$c_{\rm pen}>C_{\rm out}e^{-B}$, uniformly for $u\in\Sigma$,
\begin{align}\label{eq:visibility}
\Prob\{u\in\Omega_{\rm pen}, \mathcal D\} \le e^{-c_{\rm pen}n}
\end{align}
for all sufficiently large $n$.
\end{proposition}

\begin{proof}[Proof of Theorem~\ref{thm:main}]
Choose $B$ as in Proposition~\ref{prop:visibility}. We first record the
elementary estimate
\begin{align}\label{eq:full-dimension}
\Prob(\mathcal D^c)\le2^{n^2-N}.
\end{align}
Indeed, every proper affine hyperplane contains at most $2^{n-1}$ points of
$\Vn$. If the sample is not full dimensional, choose at most $n$ affinely
independent sample points spanning its affine hull and pad the list by
repetitions to obtain an ordered $n$-tuple of cube vertices. There are at
most $2^{n^2}$ such tuples. For each tuple whose affine span is proper, fix a
proper affine hyperplane containing that span; the probability that all $N$
samples lie in this hyperplane is at most $2^{-N}$.

For a fixed sample configuration, $\Omega_{\rm pen}$ is a finite union of
preimages of closed halfspaces under the continuous map $u\mapsto x(u)$.
Thus Tonelli's theorem and Proposition~\ref{prop:visibility} give
$$
\E[A_{\rm pen}\mathbf1_{\mathcal D}] \le e^{-c_{\rm pen}n}A_+.
$$
Together with \eqref{eq:full-dimension},
\begin{align}\label{eq:pen-area}
\E A_{\rm pen} \le(e^{-c_{\rm pen}n}+2^{n^2-N})A_+ = o(q_{B,n}A_+).
\end{align}

Let $Z_N = f_{n-1}(P_N)$ on $\mathcal D$ and $Z_N = 0$ on $\mathcal D^c$.
Since $A_{\rm out}\le A_+$, inequality \eqref{eq:facet-cover} remains true
on $\mathcal D^c$ after replacing $f_{n-1}(P_N)$ by $Z_N$. Taking
expectations and using \eqref{eq:uncovered} and \eqref{eq:pen-area},
$$
\E Z_N(6B)^{(n-1)/2}\omega_{n-1} \ge\frac12q_{B,n}A_+
$$
for all sufficiently large $n$. By \eqref{eq:outer-area},
$$
\E Z_N\ge a^n n^{(n-1)/2}
$$
for an absolute $a>0$. Set $c := a^2/4$. Since $2^n\ge\sqrt n$ for all
large $n$,
$$
(cn)^{n/2} = \frac{a^n n^{n/2}}{2^n} \le a^n n^{(n-1)/2}.
$$
Every full-dimensional sign polytope is affinely equivalent to a
full-dimensional $0/1$-polytope, so $Z_N\le g(n)$ pointwise. Hence
$g(n)\ge\E Z_N\ge(cn)^{n/2}$.
\end{proof}

\section{The outer patch}\label{sec:outer-proof}

\begin{proof}[Proof of Proposition~\ref{prop:outer}]
Choose $0<\gamma<\gamma_*$ so small that
\begin{align}\label{eq:gamma-choice}
c_{\rm MS}\lambda(\gamma)\le\frac14,
\end{align}
and set $\gamma_0 := \gamma/2$. Let $b_{\gamma_0}$ and $r$ be supplied by
Lemma~\ref{lem:geometry}. Choose $\alpha>0$ so small that
$$
4\alpha<\min\{b_{\gamma_0},\log2\}.
$$
For every fixed $B$ and all sufficiently large $n$, the levels $t_\pm$ in
\eqref{eq:levels} belong to $[\alpha n/2,2\alpha n]$. Hence
Lemma~\ref{lem:geometry}(ii) shows that $x(u)\in\gamma_0\Qn$ for
$u\in\Sigma$ and
\begin{align}\label{eq:area-exact}
S_{\cE_{t_+}}(\Sigma) \ge e^{-n/2}(1-\gamma_0^2)^{n-1}(2t_+)^{(n-1)/2}\omega_{n-1}.
\end{align}
Truncation by $\gamma\Qn$ does not alter this patch. Indeed, the patch lies
in $\gamma_0\Qn$, the new coordinate faces have normals
$\pm e_i\notin\Sigma$ for $n>r$, and all remaining intersections with the
coordinate faces are lower dimensional. Therefore
$$
S_{E_+} \restriction_\Sigma = S_{\cE_{t_+}} \restriction_\Sigma.
$$
The relevant portion of the rate boundary is smooth and strictly convex, so
its Gauss map has inverse $u\mapsto x(u)$. Consequently, for every Borel set
$U\subset\Sigma$,
\begin{align}\label{eq:normal-area}
S_{E_+}(U) = \cH^{n-1}(\{x(u):u\in U\}).
\end{align}
Since $t_+\ge\alpha n/2$, \eqref{eq:area-exact} implies
\eqref{eq:outer-area} after decreasing an absolute constant $c_0>0$.

Fix $u\in\Sigma$ and write $x = x(u)$. Smoothness and strict convexity give
$$
\nabla\rate(x) = \theta u
$$
for some $\theta>0$. Put $a_i := \lambda(x_i)$ and
$$
T := \ell_x(X) = \sum_i a_i(X_i-x_i),
$$
where $X$ has law $\Prob_x$. Since $x\in\gamma_0\Qn$ and
$\rate(x) = t_+\asymp n$, the continuous functions
$I(s)$ and $\lambda(s)^2(1-s^2)$ are comparable on
$[-\gamma_0,\gamma_0]$, and hence
\begin{align}
cn\le\sigma_x^2 := \Var_x(T) = \sum_i a_i^2(1-x_i^2)\le Cn.
\end{align}
Moreover,
$$
\sum_i\E_x|a_i(X_i-x_i)|^3 \le2\lambda(\gamma_0)\sigma_x^2\le Cn.
$$
Berry-Esseen therefore implies that every interval $J\subset\R$ of
length $h\ge0$ satisfies
\begin{align}\label{eq:interval-anti}
\Prob_x\{T\in J\}\le C\frac{1+h}{\sqrt n}.
\end{align}

Because $a = \theta u$, the tangent-cap event is $\{T\ge0\}$. By the
likelihood identity,
$$
\mu_n\{T\ge0\} = e^{-t_+}\E_x[e^{-T}\mathbf1_{\{T\ge0\}}] \le e^{-t_+}\sum_{j\ge0}e^{-j}\Prob_x\{j\le T<j+1\}.
$$
Estimate \eqref{eq:interval-anti} proves \eqref{eq:outer-cap}. Likewise,
for $1\le L\le\sqrt n$,
$$
\mu_n\{|T|\le L\} \le e^{L-t_+}\Prob_x\{|T|\le L\},
$$
and \eqref{eq:interval-anti} gives \eqref{eq:outer-window}.

Let $p_+(u)$ be the tangent-cap probability in \eqref{eq:outer-cap}. If no
sample point lies in that cap, then $x(u)\notin P_N$. For all large $n$,
$p_+(u)\le1/2$, and therefore
$$
\Prob\{x(u)\notin P_N\} \ge(1-p_+(u))^N\ge e^{-2Np_+(u)}.
$$
Since $N\le e^{\alpha n}$ and
$\alpha n-t_+ = \frac32\log n-B$,
$$
Np_+(u)\le Cne^{-B}.
$$
Integration over $\Sigma$ against $S_{E_+}$ proves
\eqref{eq:uncovered}.
\end{proof}

\section{Penetrating facets}\label{sec:penetrating}

For $\nu\in\Sn$ and $b\in\R$, write
$$
H^+(\nu,b) := \{z\in\R^n:\ip{\nu}{z}\ge b\}.
$$
A penetrating halfspace is measured by the smallest rate level it reaches
inside the truncated cube.

\begin{lemma}[Rate projection]\label{lem:projection}
Let $x\in\gamma_0\Qn$ satisfy $\rate(x) = t_x$, and suppose that
$H^+(\nu,b)$ contains $x$ and
$$
\operatorname{int}H^+(\nu,b)\cap(\cE_{t_0}\cap\gamma\Qn) \ne\varnothing \quad(0<t_0<t_x).
$$
Set
\begin{align}\label{eq:tH}
t_H := \min\{\rate(y):y\in\gamma\Qn, \ip{\nu}{y}\ge b\}.
\end{align}
Then $0\le t_H<t_0$. If $b\le0$, then $t_H = 0$. If $b>0$, the minimizer
$y$ is unique, lies on $\ip{\nu}{y} = b$, and, with
$$
J := \{i:|y_i| = \gamma\},\quad m := |J|,
$$
one has
\begin{align}
\norm{x-y}_2^2&\le2(t_x-t_H), \label{eq:contact-distance} \\
m&\le\frac{2(t_x-t_H)}{(\gamma-\gamma_0)^2}. \label{eq:active-count}
\end{align}
Moreover, for some $q>0$ and $a = q\nu$,
\begin{align}\label{eq:KKT}
a_i = \lambda(y_i)\quad(i\notin J), \quad a_iy_i>0\quad(i\in J).
\end{align}
\end{lemma}

\begin{proof}
The feasible set in \eqref{eq:tH} is compact. A point in the strict
intersection can be contracted slightly toward the origin while remaining
strictly feasible and entering the interior of $\gamma\Qn$; its rate then
becomes strictly smaller than $t_0$, unless the origin is already feasible.
Thus $t_H<t_0$, and $t_H = 0$ when $b\le0$.

Assume $b>0$. Strict convexity gives a unique minimizer $y$, and the
halfspace constraint is active. Slater's condition holds, so the KKT
conditions give
$$
q\nu = \nabla\rate(y)+\eta
$$
for some $q\ge0$ and $\eta\in N_{\gamma\Qn}(y)$. Here $\eta_i = 0$ when
$|y_i|<\gamma$, while $\eta_iy_i\ge0$ when $|y_i| = \gamma$. Necessarily
$q>0$, since otherwise $y$ would minimize $\rate$ on $\gamma\Qn$ and hence
would equal the origin. This proves \eqref{eq:KKT}.

Since $x$ is feasible, $q\nu\cdot(x-y)\ge0$. On every active coordinate,
$\eta_i(x_i-y_i)\le0$ because $x\in\gamma_0\Qn$. Hence
$\nabla\rate(y)\cdot(x-y)\ge0$. Since $\nabla^2\rate\succeq I_n$,
$$
\rate(x)\ge\rate(y)+\nabla\rate(y)\cdot(x-y) +\frac12\norm{x-y}_2^2,
$$
which proves \eqref{eq:contact-distance}. Finally,
$|x_i-y_i|\ge\gamma-\gamma_0$ for $i\in J$, giving
\eqref{eq:active-count}.
\end{proof}

For a penetrating facet visible from $u\in\Sigma$, put $x = x(u)$ and let
$t_H$ be its contact level. Define
\begin{align}\label{eq:depth}
s := t_{-} -t_H\ge0, \quad d := t_+-t_H = 2B+s.
\end{align}
If $b\le0$, put $m = 0$; otherwise let $m$ be the active-coordinate count in
Lemma~\ref{lem:projection}. Then
\begin{align}\label{eq:m-depth}
m\le C_{\rm act}d, \quad C_{\rm act} := \frac{2}{(\gamma-\gamma_0)^2}.
\end{align}
Associate with $H^+(\nu,b)$ the buffered cap
\begin{align}
C_H := \left\{z\in\Vn:\ip{\nu}{z}\ge b+\frac3n\right\}.
\end{align}

\begin{proposition}[Buffered cap]\label{prop:cap}
There are absolute constants $c,C,\delta>0$ such that, for all sufficiently
large $n$, the following holds under the assumptions of
Lemma~\ref{lem:projection}. If $b\le0$, put $m := 0$; if $b>0$, let $m$ be the
active-coordinate count from Lemma~\ref{lem:projection}. Then
\begin{align}\label{eq:cap-global}
\mu_n(C_H)\ge c\frac{e^{-t_H}}{\sqrt n (m+1)^{3/2}}.
\end{align}
Consequently, for a penetrating halfspace visible from $x(u)$,
\begin{align}\label{eq:cap-depth}
\mu_n(C_H)\ge c\frac{e^{-t_H}}{\sqrt n (1+d)^{3/2}}.
\end{align}
If, in addition,
$$
t_x\in[\alpha n/2,2\alpha n], \quad t_H\ge\alpha n/4, \quad d := t_x-t_H\le\delta n,
$$
then
\begin{align}\label{eq:cap-local}
\mu_n(C_H\cap\{|\ell_x|\le C(1+d)\}) \ge c\frac{e^{-t_H}}{\sqrt n (1+d)^{3/2}}.
\end{align}
\end{proposition}

We also need two standard empty-range estimates. Their short proofs are
included in Appendix~\ref{sec:sampling}.

\begin{lemma}[Empty-range estimates]\label{lem:empty-range}
\emph{(i)} Let $(\Omega,\mu)$ be finite, let $W\subset\Omega$ have measure
$0<q\le\bar q\le1$, and let $\mathcal R$ be a class of subsets of $W$ of
VC dimension at most $v\ge1$. Suppose $\mu(R)\ge p>0$ for every
$R\in\mathcal R$. If $Y_1,\dots,Y_N$ are independent with law $\mu$ and
$Np\ge16v$, then
\begin{align}
&\Prob\{\exists R\in\mathcal R: R\cap\{Y_1,\dots,Y_N\} = \varnothing\} \label{eq:conditional-net} \\
&\quad\le2e^{-Np/12} +2\exp\left[v\log\left(\frac{4eN\bar q}{v}\right)-c_{\rm VC}Np\right], \notag
\end{align}
where $c_{\rm VC}>0$ is absolute.

\emph{(ii)} Let $p>0$. The probability that there exist $\nu\in\Sn$ and
$b\in[-\sqrt n,\sqrt n]$ such that
$$
\ip{\nu}{X^{(j)}}\le b\quad(1\le j\le N), \quad \mu_n\left\{z:\ip{\nu}{z}\ge b+\frac3n\right\}\ge p
$$
is at most
\begin{align}\label{eq:halfspace-net}
Cn^{3/2}(9n^{3/2})^ne^{-Np}.
\end{align}
\end{lemma}

\begin{proof}[Proof of Proposition~\ref{prop:visibility}]
Fix $u\in\Sigma$ and put $x = x(u)$. A penetrating facet visible from $u$
has an outer halfspace $H^+(\nu,b)$ whose buffered cap $C_H$ contains no
sample point. By \eqref{eq:levels} and \eqref{eq:depth},
\begin{align}
t_H = \alpha n-\frac32\log n-B-s, \quad d = 2B+s.
\end{align}
Set $K_n := \lceil2\log\log n\rceil$.

Suppose first that $0\le s<K_n$. For an integer $0\le k<K_n$, consider
the shell $k\le s<k+1$ and put
$$
A_k := 2B+k+2, \quad L_k := C_1A_k, \quad W_k := \{z\in\Vn:|\ell_x(z)|\le L_k\},
$$
where $C_1\ge1$ is at least the constant in \eqref{eq:cap-local}.
Uniformly in
these shells, $L_k = o(\sqrt n)$, $d = o(n)$, and $t_H\ge\alpha n/4$ for all
large $n$. Proposition~\ref{prop:cap} and
Proposition~\ref{prop:outer} therefore give
\begin{align}
\mu_n(C_H\cap W_k)&\ge p_k, &p_k& := c_1\frac{e^{-t_{-}+k}}{\sqrt n A_k^{3/2}}, \\
\mu_n(W_k)&\le\bar q_k, &\bar q_k& := C_1'\frac{A_ke^{C_1A_k-t_+}}{\sqrt n}.
\end{align}
Here $t_H = t_{-} -s$, $s\ge k$, and $1+d\le A_k$. Also
$\bar q_k<1$ for all large $n$, uniformly in $k<K_n$, because
$A_k = O(B+\log\log n)$ and
$$
\log\bar q_k\le-\alpha n+\log n+O(B+\log\log n).
$$

For the remainder of the shallow-shell argument, $u$ and $k$ are fixed;
therefore $x$, $W_k$, and the range class below are deterministic. Let
$$
\mathcal R_k := \{W_k\cap H:H\text{ is an affine halfspace and } \mu_n(W_k\cap H)\ge p_k\}.
$$
A penetrating facet in the shell produces an empty member of
$\mathcal R_k$, namely
$$
C_H\cap W_k = W_k\cap\left\{z:\ip{\nu}{z}\ge b+\frac3n\right\}.
$$
Since affine halfspaces in $\R^n$ have VC dimension $n+1$,
$\operatorname{VCdim}(\mathcal R_k)\le n+1$. From
$N\ge\frac12e^{\alpha n}$ and the definitions of $t_\pm$,
\begin{align}
Np_k&\ge a_0n\frac{e^{B+k}}{A_k^{3/2}}, \\
N\bar q_k&\le C_2nA_ke^{C_1A_k-B}.
\end{align}
Consequently,
\begin{align}
\log\left(\frac{4eN\bar q_k}{n+1}\right)\le C_3A_k.
\end{align}
Choose $B$ sufficiently large that, for every $k\ge0$,
\begin{align}\label{eq:B-choice}
a_0\frac{e^{B+k}}{A_k^{3/2}}\ge32, \quad c_{\rm VC}a_0\frac{e^{B+k}}{A_k^{3/2}}\ge4C_3A_k.
\end{align}
This is possible because $e^{B+k}/(2B+k+2)^{5/2}$ is increasing in $k$
for large $B$, and its value at $k = 0$ tends to infinity with $B$.
Let $\mathcal A_k(u)$ be the event that $\mathcal D$ occurs and a
penetrating facet in shell $k$ is visible from $u$. If
$\mathcal R_k\ne\varnothing$, then
$p_k\le\mu_n(W_k)\le\bar q_k$; otherwise $\mathcal A_k(u)$ is empty.
Thus Lemma~\ref{lem:empty-range}(i), with $v = n+1$, gives
$$
\Prob(\mathcal A_k(u)) \le4\exp\left[-a_1n\frac{e^{B+k}}{A_k^{3/2}}\right].
$$
Let $\mathcal A_{\rm sh}(u) := \bigcup_{k<K_n}\mathcal A_k(u)$. Summing
over $k<K_n$ gives
\begin{align}\label{eq:shallow}
\Prob(\mathcal A_{\rm sh}(u)) \le\exp\left[-a_2n\frac{e^B}{B^{3/2}}\right].
\end{align}

Now suppose $s\ge K_n$. By \eqref{eq:cap-depth},
$$
N\mu_n(C_H) \ge cn\frac{e^{B+s}}{(1+B+s)^{3/2}}.
$$
The right-hand side is increasing in $s$ once $B$ is large. Since
$e^{K_n}\ge(\log n)^2$,
\begin{align}
N\mu_n(C_H) \ge\Lambda_n := cn\frac{e^B(\log n)^2}{(\log\log n)^{3/2}}.
\end{align}
Let $\mathcal A_{\rm deep}(u)$ be the event that $\mathcal D$ occurs
and a deep penetrating facet is visible from $u$. For a facet halfspace,
$|b|\le\norm\nu_1\le\sqrt n$. Applying
Lemma~\ref{lem:empty-range}(ii) with $p = \Lambda_n/N$ gives
$$
\Prob(\mathcal A_{\rm deep}(u)) \le Cn^{3/2}(9n^{3/2})^ne^{-\Lambda_n}.
$$
Since the logarithm of the prefactor is
$\frac32n\log n+O(n)$ while $\Lambda_n\gg n\log n$,
\begin{align}\label{eq:deep}
\Prob(\mathcal A_{\rm deep}(u)) \le\exp\left[-cn\frac{(\log n)^2}{(\log\log n)^{3/2}}\right].
\end{align}

The constant $a_2$ in \eqref{eq:shallow} is absolute and independent of
$B$. Increase $B$, if necessary before fixing it, so that in addition to
\eqref{eq:B-choice},
$$
\frac{a_2}{4}\frac{e^B}{B^{3/2}}>C_{\rm out}e^{-B}.
$$
For all sufficiently large $n$, the deep bound \eqref{eq:deep} is at most
$$
\exp\left[-\frac{a_2}{2}n\frac{e^B}{B^{3/2}}\right].
$$
Combining the shallow and deep estimates gives
$$
\Prob\{u\in\Omega_{\rm pen}, \mathcal D\} \le 2\exp\left[-\frac{a_2}{2}n\frac{e^B}{B^{3/2}}\right].
$$
For all sufficiently large $n$,
$$
2\exp\left[-\frac{a_2}{2}n\frac{e^B}{B^{3/2}}\right] \le \exp\left[-\frac{a_2}{4}n\frac{e^B}{B^{3/2}}\right].
$$
Thus \eqref{eq:visibility} holds with
$$
c_{\rm pen} := \frac{a_2}{4}\frac{e^B}{B^{3/2}} >C_{\rm out}e^{-B}.
$$
\end{proof}

\section{Proof of the buffered-cap estimate}\label{sec:cap-proof}

We prove Proposition~\ref{prop:cap}. The argument separates the coordinates
at the rate projection into an interior block, where the KKT normal equals
the natural parameter, and an active block, where the weights are arbitrary
but all coordinates are pinned at $\pm\gamma$.

\paragraph{Claim 1: active coordinates.}
For every $\bar\gamma\in(0,1)$ there is $c(\bar\gamma)>0$ such that, for
all $m\ge1$, $q\in[0,\bar\gamma]$, and positive weights
$s_1,\dots,s_m$,
\begin{align}\label{eq:active-tail}
\mu_m\left\{\sum_{i = 1}^m s_i(X_i-q)\ge0\right\} \ge c(\bar\gamma)(m+1)^{-3/2}e^{-mI(q)}.
\end{align}

\begin{proof}
Let $\tau$ be cyclic shift on $\{-1,1\}^m$ and put
$L(z) := \sum_i s_i(z_i-q)$. Set
$$
k := \left\lceil\frac{1+q}{2}m\right\rceil, \quad \mathcal A := \{z:|\{i:z_i = 1\}|\ge k\}.
$$
If $z\in\mathcal A$, then
$$
\sum_{r = 0}^{m-1}L(\tau^rz) = \left(\sum_i s_i\right)\left(\sum_j(z_j-q)\right)\ge0.
$$
Thus at least one point of every cyclic orbit meeting $\mathcal A$ satisfies
$L\ge0$, and every orbit has at most $m$ points. Hence
$$
\mu_m\{L\ge0\} \ge\frac1m\mu_m(\mathcal A) \ge\frac1m2^{-m}\binom{m}{k}.
$$
For all sufficiently large $m$, uniformly in $q\in[0,\bar\gamma]$,
Stirling's formula gives
$$
2^{-m}\binom{m}{k} \ge c_{\bar\gamma}m^{-1/2}e^{-mI(q_m)}, \quad q_m := \frac{2k}{m}-1.
$$
Since $0\le q_m-q<2/m$ and $I' = \lambda$ is bounded on a fixed interval
containing $[0,\bar\gamma]$, one has
$m(I(q_m)-I(q))\le C_{\bar\gamma}$. This proves
\eqref{eq:active-tail} for large $m$. The finitely many remaining values
are absorbed into the constant, using the all-positive sign vector.
\end{proof}

\paragraph{Claim 2: shifted block estimates.}
For each fixed $L>0$ there are $c_L>0$ and $n_L$ such that, for
$n\ge n_L$, the following hold.

If $0\le r\le n$, $y\in(-\gamma,\gamma)^r$,
$a_i = \lambda(y_i)$, and $E = \sum_iI(y_i)$, then
\begin{align}\label{eq:shifted-interior}
\mu_r\left\{\sum_{i = 1}^r a_i(X_i-y_i) \ge\frac{L}{n}\norm a_2\right\} \ge c_L\frac{e^{-E}}{\sqrt{1+E}}.
\end{align}
If $0\le m\le n$ and $s_1,\dots,s_m>0$, then
\begin{align}\label{eq:shifted-active}
\mu_m\left\{\sum_{i = 1}^m s_i(X_i-\gamma) \ge\frac{L}{n}\norm s_2\right\} \ge c_L(m+1)^{-3/2}e^{-mI(\gamma)}.
\end{align}
The assertions are interpreted as probability one for an empty block.

\begin{proof}
We first prove \eqref{eq:shifted-interior}. After changing coordinate signs,
assume $y_i,a_i\ge0$. The assertion is immediate if $a = 0$. Under
$\Prob_y$, let
$$
T := \sum_i a_i(X_i-y_i), \quad \sigma^2 := \Var_y(T) = \sum_i a_i^2(1-y_i^2).
$$
On $[-\gamma,\gamma]$,
$I(y)$ and $\lambda(y)^2(1-y^2)$ are comparable, so
\begin{align}\label{eq:energy-variance}
c_\gamma E\le\sigma^2\le C_\gamma E.
\end{align}
Also
$$
\sum_i\E_y|a_i(X_i-y_i)|^3 \le2\lambda(\gamma)\sigma^2.
$$
Let $\ell := (L/n)\norm a_2$. Since $r\le n$, one has
$0\le\ell\le L\lambda(\gamma)/\sqrt n\le1$ for large $n$.
Berry-Esseen and \eqref{eq:energy-variance} give a constant
$C_\gamma$ such that the approximation error is at most
$C_\gamma/\sigma$. Choose a fixed $A>0$ so large that the Gaussian mass of
an interval of normalized length $A/\sigma$ near the origin dominates twice
this error, and then choose $E_*$ so large that $E\ge E_*$ implies
$\sigma\ge2(A+1)$. Since $0\le\ell\le1$, for $E\ge E_*$ we obtain
$$
\Prob_y\{\ell\le T\le\ell+A\}\ge\frac c\sigma.
$$
The likelihood identity then yields
$$
\mu_r\{T\ge\ell\} \ge e^{-E-\ell-A}\Prob_y\{\ell\le T\le\ell+A\} \ge c\frac{e^{-E}}{\sqrt E}.
$$

It remains to treat $0<E<E_*$. The desired event is
$\sum_i a_iX_i\ge h$, where
$$
h := \sum_i a_iy_i+\frac{L}{n}\norm a_2 \le\norm a_2^2+\frac{L}{n}\norm a_2
$$
because $0\le y_i\le a_i$. Put
$$
R := \frac{\norm a_2}{\norm a_\infty}, \quad \tau := c_{\rm MS}\frac h{\norm a_2}.
$$
By \eqref{eq:gamma-choice},
$$
\frac\tau R \le c_{\rm MS}\norm a_\infty+\frac{c_{\rm MS}L}{n}<1
$$
for large $n$, and
$\tau^2\le C(\norm a_2^2+n^{-2})\le C(E+n^{-2})$.
Since $c_{\rm MS}^{-1}\tau\norm a_2 = h$, the
Montgomery-Smith estimate \eqref{eq:MS} gives a positive lower bound
depending only on $L$ and $\gamma$. This proves
\eqref{eq:shifted-interior} after adjusting $c_L$.

For \eqref{eq:shifted-active}, assume $m\ge1$, put
$\ell = (L/n)\norm s_2$, and define
$$
\gamma' := \gamma+\frac{\ell}{\sum_i s_i}.
$$
Since $\norm s_2\le\sum_i s_i$, one has
$0\le\gamma'-\gamma\le L/n$. For a fixed
$\bar\gamma\in(\gamma,1)$ and all large $n$, Claim~1 gives
$$
\mu_m\left\{\sum_i s_i(X_i-\gamma)\ge\ell\right\} \ge c(m+1)^{-3/2}e^{-mI(\gamma')}.
$$
Because $m\le n$ and $I' = \lambda$,
$m(I(\gamma')-I(\gamma))\le L\lambda(\bar\gamma)$, which proves
\eqref{eq:shifted-active}.
\end{proof}

\paragraph{Claim 3: joint likelihood localization.}
Fix $\kappa>0$. There are constants $L_0,K,c>0$, depending only on
$\gamma$ and $\kappa$, such that the following holds for all sufficiently
large $n$. Let $r\le n$, $x,y\in(-\gamma,\gamma)^r$, and put
$$
a_i := \lambda(y_i),\quad w_i := \lambda(x_i)-\lambda(y_i),
$$
$$
T := \sum_i a_i(X_i-y_i),\quad W := \sum_i w_i(X_i-y_i).
$$
If
$$
\sum_iI(y_i)\ge\kappa n, \quad \rho\ge0, \quad \norm w_2^2\le\rho, \quad 0\le\ell\le1,
$$
then
\begin{align}\label{eq:joint}
\Prob_y\{\ell\le T\le\ell+L_0, |W|\le K(1+\rho)\} \ge\frac c{\sqrt n}.
\end{align}

\begin{proof}
Since $|\lambda(y_i)|\le\lambda(\gamma)$ and
$I(\tanh t)\asymp_\gamma t^2$ on this interval, the energy assumption gives
$$
c_{\gamma,\kappa}n\le\sum_i a_i^2\le C_\gamma n.
$$
Under $\Prob_y$, the centered variable $T$ has variance comparable to $n$
and total third absolute moment $O(n)$. Berry-Esseen therefore gives, for
$L_0$ sufficiently large,
\begin{align}\label{eq:T-window}
\Prob_y\{\ell\le T\le\ell+L_0\}\ge\frac{c_0}{\sqrt n}.
\end{align}

For $\tau\in\{-1,1\}$, let $\Prob_\tau$ be the product law with natural
parameters $a_i+\tau w_i$, and let
$m_{i,\tau} := \tanh(a_i+\tau w_i)$. For $\tau = 1$ these parameters are
$\lambda(x_i)$, while for $\tau = -1$ they are
$2\lambda(y_i)-\lambda(x_i)$; hence their absolute values are at most
$3\lambda(\gamma)$. Under $\Prob_\tau$,
$$
\Var_\tau(T) = \sum_i a_i^2(1-m_{i,\tau}^2)\asymp n, \quad \sum_i\E_\tau|a_i(X_i-m_{i,\tau})|^3 = O(n).
$$
Berry-Esseen under $\Prob_\tau$ therefore shows that every interval $J$
of length $L_0$ satisfies
\begin{align}\label{eq:secondary-window}
\Prob_\tau\{T\in J\}\le\frac C{\sqrt n}.
\end{align}
Taylor's theorem for $\log\cosh$, whose second derivative is at most one,
gives
$$
\Psi(\tau) := \sum_i( \log\cosh(a_i+\tau w_i)-\log\cosh a_i -\tau w_i\tanh a_i) \le\frac12\norm w_2^2\le\frac\rho2.
$$
For $A := \{\ell\le T\le\ell+L_0\}$, exponential tilting and
\eqref{eq:secondary-window} yield
$$
\E_y[e^{\tau W}\mathbf1_A] = e^{\Psi(\tau)}\Prob_\tau(A) \le\frac C{\sqrt n}e^{\rho/2}.
$$
Markov's inequality for both signs of $\tau$ gives
$$
\Prob_y\{A,|W|>K(1+\rho)\} \le\frac{2C}{\sqrt n}e^{-K-(K-1/2)\rho}.
$$
Choosing $K$ large enough and using \eqref{eq:T-window} proves
\eqref{eq:joint}.
\end{proof}

\begin{proof}[Proof of Proposition~\ref{prop:cap}]
If $b\le0$, then $t_H = 0$ and $m = 0$. For
$S := \ip{\nu}{X}$, symmetry gives $\E S^2 = 1$, while independence gives
$\E S^4\le3$. Paley-Zygmund applied to $S^2$ yields
$\Prob\{S\ge1/2\}\ge3/32$. For all large $n$, this event lies in $C_H$,
which proves \eqref{eq:cap-global} in this case.

Assume $b>0$, and let $y,J,m,a$ be given by
Lemma~\ref{lem:projection}. Put
$$
I_0 := [n]\setminus J, \quad E_0 := \sum_{i\in I_0}I(y_i).
$$
For each coordinate, let $\varepsilon_i := \operatorname{sgn}(y_i)$, with
$\operatorname{sgn}(0) := 1$, and put $\widetilde X_i := \varepsilon_iX_i$.
The variables $\widetilde X_i$ are again independent uniform signs. On
$I_0$, the KKT relation gives
$$
a_i(X_i-y_i) = |a_i|(\widetilde X_i-|y_i|),
$$
whereas on $J$ one has $y_j = \varepsilon_j\gamma$ and
$$
a_j(X_j-y_j) = |a_j|(\widetilde X_j-\gamma).
$$
Since $a = q\nu$, $a\cdot y = qb$, and $\norm a_2 = q$, the cap event is
$$
C_H = \left\{a\cdot(X-y)\ge\frac3n\norm a_2\right\}.
$$
It contains the intersection of the two independent events
\begin{align*}
\sum_{i\in I_0}|a_i|(\widetilde X_i-|y_i|) &\ge\frac3n\norm{a_{I_0}}_2, \\
\sum_{j\in J}|a_j|(\widetilde X_j-\gamma) &\ge\frac3n\norm{a_J}_2,
\end{align*}
because
$\norm{a_{I_0}}_2+\norm{a_J}_2\ge\norm a_2$.
Claim~2, with $L = 3$, gives
$$
\mu_n(C_H) \ge c\frac{e^{-E_0}}{\sqrt{1+E_0}} (m+1)^{-3/2}e^{-mI(\gamma)}.
$$
Since $E_0+mI(\gamma) = \rate(y) = t_H$ and
$1+E_0\le1+n\log2\le Cn$, this proves \eqref{eq:cap-global}.
Estimate \eqref{eq:cap-depth} follows from \eqref{eq:m-depth}.

It remains to prove the localized estimate. The assumptions imply $b>0$.
Choose $\delta>0$ so small that
$C_{\rm act}\delta I(\gamma)\le\alpha/8$. Then, by
\eqref{eq:active-count},
\begin{align}
E_0 = t_H-mI(\gamma)\ge\frac\alpha8n.
\end{align}
On $I_0$, put
$$
a_i := \lambda(y_i), \quad w_i := \lambda(x_i)-\lambda(y_i),
$$
$$
T_0 := \sum_{i\in I_0}a_i(X_i-y_i), \quad W_0 := \sum_{i\in I_0}w_i(X_i-y_i).
$$
Since $\lambda'$ is bounded on $[-\gamma,\gamma]$,
\eqref{eq:contact-distance} gives
\begin{align}
\norm{w_{I_0}}_2^2\le C\norm{x-y}_2^2\le Cd.
\end{align}
Since $\norm{a_{I_0}}_2\le\lambda(\gamma)\sqrt n$, for all sufficiently
large $n$ the number $\ell_0 := 3\norm{a_{I_0}}_2/n$ belongs to $[0,1]$.
Applying Claim~3 with
$\kappa = \alpha/8$ and $\rho = Cd$, we obtain an event of
$\Prob_y^{I_0}$-measure at least $c/\sqrt n$ on which
\begin{align}\label{eq:local-event}
\ell_0\le T_0\le\ell_0+L_0, \quad |W_0|\le C(1+d).
\end{align}
On this event $T_0\le1+L_0$, so the likelihood identity shows that its
$\mu_{I_0}$-measure is at least $ce^{-E_0}/\sqrt n$. By
\eqref{eq:shifted-active}, the independent active block satisfies
$$
\mu_J\left\{\sum_{j\in J}a_j(X_j-y_j) \ge\frac3n\norm{a_J}_2\right\} \ge c(m+1)^{-3/2}e^{-mI(\gamma)}.
$$
Intersecting these two events and using $m\le C_{\rm act}d$ produces a
subset of $C_H$ of measure at least
\begin{align}\label{eq:local-mass}
c\frac{e^{-t_H}}{\sqrt n (1+d)^{3/2}}.
\end{align}

We finally locate this subset in the outer likelihood window. The Bregman
divergence
$$
D_{\rate}(y,x) := \rate(y)-\rate(x) -\nabla\rate(x)\cdot(y-x)
$$
satisfies, because $x,y\in\gamma\Qn$,
$$
0\le D_{\rate}(y,x) \le\frac1{2(1-\gamma^2)}\norm{x-y}_2^2\le Cd.
$$
Since $\rate(x)-\rate(y) = d$,
\begin{align}\label{eq:contact-score}
|\nabla\rate(x)\cdot(y-x)| = d+D_{\rate}(y,x)\le C(1+d).
\end{align}
On the interior block, \eqref{eq:local-event} gives
$$
\sum_{i\in I_0}\lambda(x_i)(X_i-y_i) = T_0+W_0 = O(1+d).
$$
On the active block,
$|\lambda(x_j)|\le\lambda(\gamma_0)$ and $m\le C_{\rm act}d$, so
$$
\left|\sum_{j\in J}\lambda(x_j)(X_j-y_j)\right|\le C(1+d).
$$
Together with \eqref{eq:contact-score}, this yields
$|\ell_x(X)|\le C(1+d)$ on the set of mass
\eqref{eq:local-mass}, proving \eqref{eq:cap-local}.
\end{proof}

\appendix
\section{Sampling estimates}\label{sec:sampling}

\begin{proof}[Proof of Lemma~\ref{lem:empty-range}(i)]
If $\mathcal R = \varnothing$, there is nothing to prove. Otherwise $p\le q$.
Let
$$
M := |\{j:Y_j\in W\}|.
$$
Then $M$ is binomial with mean $Nq$, and Chernoff's inequality gives
$$
\Prob\{M\notin[Nq/2,2Nq]\} \le2e^{-Nq/12}\le2e^{-Np/12}.
$$
Condition on $M = m\in[Nq/2,2Nq]$ and on the set of indices for which
$Y_j\in W$. The corresponding observations are independent with common law
$\nu := \mu(\cdot\mid W)$. Every $R\in\mathcal R$ has
$\nu(R)\ge\varepsilon := p/q$, and
$\varepsilon m\ge Np/2\ge8v$.

Because $\Omega$ is finite, $\mathcal R$ has only finitely many distinct
members; fix an ordering of them. Let $S$ be this conditional sample and
$S'$ an independent sample of size $m$ from $\nu$. If $S$ misses a range
of $\nu$-measure at least $\varepsilon$, choose the first such range in
that ordering. Conditional
on $S$, Chernoff's inequality shows that $S'$ contains at least
$\varepsilon m/2$ points of that range with probability greater than $1/2$.
Thus
$$
\Prob\{S\text{ misses a heavy range}\} \le2\Prob\{\exists R:S\cap R = \varnothing, |S'\cap R|\ge\varepsilon m/2\}.
$$
Pool the two indexed samples. Conditional on the pooled sample, a fixed trace
of size $r\ge\varepsilon m/2$ lies entirely in the ghost half with
probability zero if $r>m$; if $r\le m$, that probability is at most
$$
\frac{(m)_r}{(2m)_r}\le2^{-r}\le e^{-c\varepsilon m}.
$$
By Sauer's lemma, the number of traces on the pooled $2m$ observations is at
most
$$
\left(\frac{2em}{v}\right)^v \le\left(\frac{4eN\bar q}{v}\right)^v.
$$
A union bound, followed by the exceptional probability for $M$, proves
\eqref{eq:conditional-net}.
\end{proof}

\begin{proof}[Proof of Lemma~\ref{lem:empty-range}(ii)]
Choose a Euclidean $1/(4n^{3/2})$-net $\cN\subset\Sn$ with
$|\cN|\le(9n^{3/2})^n$, and discretize $b$ in steps of $1/n$. If
$\norm{\nu-\nu'}_2\le1/(4n^{3/2})$ and
$b\in[j/n,(j+1)/n)$, then for every $z\in\Vn$,
$$
\left\{\ip{\nu}{z}\ge b+\frac3n\right\} \subset \left\{\ip{\nu'}{z}\ge\frac{j+2}{n}\right\} \subset \{\ip{\nu}{z}>b\}.
$$
Thus an empty buffered cap of measure at least $p$ yields an empty member of
a discretized family of at most
$Cn^{3/2}(9n^{3/2})^n$ halfspaces, each of measure at least $p$. The
emptiness probability for each fixed member is at most $e^{-Np}$. A union
bound proves \eqref{eq:halfspace-net}.
\end{proof}

\section*{Statements and Declarations}

The author has no competing interests to declare that are relevant to the content of this article. Data sharing is not applicable to this article as no datasets were generated or analysed during the current study. No funding was received for conducting this study.


\begin{thebibliography}{99}

\bibitem{BP}
I.~B\'ar\'any and A.~P\'or,
\emph{On $0$-$1$ polytopes with many facets},
Adv. Math. \textbf{161} (2001), no.~2, 209-228.

\bibitem{CF26}
F.~Castillo and L.~Ferroni,
\emph{Almost factorial many facets for $0/1$-polytopes},
arXiv:2608.27247 (2026).

\bibitem{DFM}
M.~E. Dyer, Z.~F\"uredi, and C.~McDiarmid,
\emph{Volumes spanned by random points in the hypercube},
Random Structures Algorithms \textbf{3} (1992), no.~1, 91-106.

\bibitem{FKR}
T.~Fleiner, V.~Kaibel, and G.~Rote,
\emph{Upper bounds on the maximal number of facets of $0/1$-polytopes},
European J. Combin. \textbf{21} (2000), no.~1, 121-130.

\bibitem{GGM05}
D.~Gatzouras, A.~Giannopoulos, and N.~Markoulakis,
\emph{Lower bound for the maximal number of facets of a $0/1$ polytope},
Discrete Comput. Geom. \textbf{34} (2005), no.~2, 331-349.

\bibitem{GGM07}
D.~Gatzouras, A.~Giannopoulos, and N.~Markoulakis,
\emph{On the maximal number of facets of $0/1$ polytopes},
in \emph{Geometric Aspects of Functional Analysis}, Lecture Notes in Math.
\textbf{1910}, Springer, Berlin, 2007, 117-125.

\bibitem{HW}
D.~Haussler and E.~Welzl,
\emph{$\varepsilon$-nets and simplex range queries},
Discrete Comput. Geom. \textbf{2} (1987), 127-151.

\bibitem{MS}
S.~J. Montgomery-Smith,
\emph{The distribution of Rademacher sums},
Proc. Amer. Math. Soc. \textbf{109} (1990), no.~2, 517-522.

\bibitem{Petrov}
V.~V. Petrov,
\emph{Limit Theorems of Probability Theory: Sequences of Independent
Random Variables},
Oxford Studies in Probability, vol.~4, Clarendon Press, Oxford, 1995.

\bibitem{Ziegler}
G.~M. Ziegler,
\emph{Lectures on $0/1$-polytopes},
in \emph{Polytopes-Combinatorics and Computation},
DMV Sem., vol.~29, Birkh\"auser, Basel, 2000, 1-41.

\end{thebibliography}
\end{document}